\documentclass[aps,tightenlines,10pt,superscriptaddress]{revtex4}%
\usepackage{amssymb}
\usepackage{amsfonts}
\usepackage{amsmath}
\usepackage{amsmath}
\usepackage{pstricks}
\usepackage{graphicx}
\providecommand{\U}[1]{\protect\rule{.1in}{.1in}}
\newtheorem{theorem}{Theorem}

\newtheorem{proposition}[theorem]{Proposition}

\newenvironment{proof}[1][Proof]{\noindent\textbf{#1.} }{\ \rule{0.5em}{0.5em}}
\begin{document}
\title{Interpreting the von Neumann entropy of graph Laplacians, and coentropic graphs}
\author{Niel de Beaudrap}
\affiliation{Department of Applied Mathematics and Theoretical
Physics, University of Cambridge, Wilberforce Road, Cambridge, United Kingdom}
\author{Vittorio Giovannetti}
\affiliation{NEST, Scuola Normale Superiore and Istituto Nanoscienze-CNR, Piazza dei Cavalieri 7, I-56126 Pisa, Italy}
\author{Simone Severini}
\affiliation{Department of Computer Science, and Department of Physics \& Astronomy,
University College London, WC1E 6BT London, United Kingdom}
\author{Richard Wilson}
\affiliation{Department of Computer Science, University of York,
Deramore Lane, Heslington, York, YO10 5GH, United Kingdom}

\begin{abstract}
For any graph, we define a rank-1 operator on a bipartite
tensor product space, with components associated to the set of vertices and
edges respectively. We show that the partial traces of the operator are the
Laplacian and the edge-Laplacian. This provides an interpretation of the von
Neumann entropy of the (normalized)\ Laplacian as the amount of quantum
entanglement between two systems corresponding to vertices and edges. In this
framework, cospectral graphs correspond exactly to local unitarily equivalent pure
states. Finally, we introduce the notion of coentropic graphs, that is, graphs with equal von Neumann entropy.
The smallest coentropic (but not cospectral) graphs that we are able to construct have $8$ vertices.
The number of equivalence classes of coentropic graphs with $n$ vertices and $m$ edges is a lower bound to the number of (pure) bipartite entanglement classes with subsystems of corresponding dimension.

\end{abstract}
\maketitle


Our references on algebraic graph theory and quantum information are
\cite{god} and \cite{nc}, respectively. Let $G=(V,E)$ be a simple graph with
$n$ vertices and $m$ edges. We consider a representation of
graphs using certain vectors in a bipartite tensor product space.
Specifically, we define two configuration (Hilbert) spaces: $\mathcal{H}%
_{V}\cong\mathbb{C}^{V}$ with orthonormal basis $\mathbf{a}_{v}$ running over
$v\in V$; $\mathcal{H}_{E}\cong\mathbb{C}^{E}$ with orthonormal basis
$\mathbf{b}_{e}$ running over $e\in E$. We assume a commutative formal product
from pairs of vertices to unordered pairs, $uv=\left\{  u,v\right\}  $. We
also assume without loss of generality that there are total orders defined on
$V$ and $E$ (which we denote by $\leq$ in both cases).


The graph representation which we consider is a vector in $\mathcal{H}_{V}
\otimes\mathcal{H}_{E}$ related to graph Laplacians. The \emph{Laplacian} of
the graph $G$ is an operator
\begin{equation}
L(G)=D(G)-A(G)
\end{equation}
acting on $\mathcal{H}_{V}$, where $D(G)=\;$Diag$\left(  \deg(v_{1}%
),\deg(v_{2}),\ldots,\deg(v_{n})\right)  $ is the \emph{degree matrix} of $G$
and $A(G)$ is the adjacency matrix.
There is an equivalent definition of a Laplacian, in terms of incidence
matrices. An \emph{orientation} of $G$ is a collection $\mathcal{F}=\left\{
f_{e}\right\}  _{e\in E}$ of bijections $f_{e}:e\longrightarrow\left\{
1,-1\right\}  $ such that that for each $uv\in E$, we have $f_{uv}%
(u)=-f_{uv}(v)\in\left\{  1,-1\right\}  $. For an orientation $\mathcal{F}$ of
$G$, the \emph{incidence matrix} of $G$ is the $n\times m$ matrix
$M_{\mathcal{F}}$ --- with rows indexed by vertices and columns by edges ---
such that
\begin{equation}
\left(  M_{\mathcal{F}}\right)  _{v,e}=%
\begin{cases}
f_{e}(v), & \text{if $v\in e;$}\\
0, & \text{otherwise}.
\end{cases}
\label{eqn:defIncidence}%
\end{equation}
Independently of the chosen orientation $\mathcal{F}$, we then have
\begin{equation}
\label{eqn:laplacianOrientationFormulation}
L(G)=M_{\mathcal{F}}M_{\mathcal{F}}^{\dagger}.
\end{equation}
This shows that $L\left(  G\right)  $ is positive semidefinite. We may
equivalently formulate the incidence matrix as a sum of outer products,
\begin{equation}
M_{\mathcal{F}}=\sum_{e\in E}\sum_{v\in e}\,f_{e}(v)\mathbf{a}_{v}%
\mathbf{b}_{e}^\dagger=\sum_{uv\in E}f_{uv}(u)(\mathbf{a}_{u}-\mathbf{a}%
_{v})\mathbf{b}_{uv}^{\dagger};
\end{equation}
the middle expression is just another presentation of the definition in Eq.
(\ref{eqn:defIncidence}); the right-hand expression follows from
$f_{uv}(u)=-f_{uv}(v)$.

Even though the Laplacian itself is the same for all orientations of $G$,
the formulation of the Laplacian in Eq.~\eqref{eqn:laplacianOrientationFormulation}
is orientation-dependent, essentially because we are not considering the graph $G$
but rather a digraph $D$ such that $A(G)=A(D)+A(D^{T})$.
Considering incidence matrices as a property of \emph{digraphs} motivates the
following definition. The \emph{incidence matrix} of a directed graph $G$ is
the $n\times m$ matrix $\bar{M}$ --- with rows indexed by vertices and columns
by arcs --- such that
\[
\bar{M}_{v,\alpha}=%
\begin{cases}
+1, & \text{if $v$ is a source of $\alpha;$}\\
-1, & \text{if $v$ is a sink of $\alpha;$}\\
0, & \text{otherwise}.
\end{cases}
\]
The matrix $M_{\mathcal{F}}$ is then the resulting \emph{directed} incidence
matrix $\bar{M}$ for the directed graph in which we replace $uv$ with the arc
$u\rightarrow v$ if $f_{uv}(u)=+1$, and with the arc $v\rightarrow u$ if
$f_{uv}(u)=-1$ (\emph{i.e.}, taking $\mathcal{F}$ literally as a specification
of how to uniquely orient the edges of $G$). Having a definition of incidence
matrices on digraphs, we can describe the Laplacian of $G$ in terms of the
incidence matrix of $G$, interpreted as a symmetric digraph containing both
the arc $u \rightarrow v$ and the arc $v \rightarrow u$ for each edge $uv \in
E(G)$. We replace each $\mathbf{b}_{uv}$ with two vectors $\mathbf{d}_{u,v}$
and $\mathbf{d}_{v,u}$ corresponding to the arcs $u\rightarrow v$ and
$v\rightarrow u$. For instance, we may do this by redefining $\mathcal{H}_{E}
\;=\; \text{span }\bigl\{ \mathbf{d}_{u,v} : uv \in E(G) \bigr\} \subseteq
\mathcal{H}_{V} \otimes\mathcal{H}_{V}$, letting $\mathbf{d}_{u,v} =
\mathbf{a}_{u} \otimes\mathbf{a}_{v}$. We thereby obtain
\begin{subequations}
\label{eqn:undirIncidence}
\begin{align}
\bar{M}  :=\sum_{u\in V}\sum_{uv\in E}\mathbf{a}_{u}\left(  \mathbf{d}%
_{u,v}-\mathbf{d}_{v,u}\right)  ^{\dagger} &=\sum_{u\in V}\sum_{uv\in E}\left(
\mathbf{a}_{u}-\mathbf{a}_{v}\right)  \mathbf{d}_{v,u}^{\dagger}%
\label{eqn:directedIncidenceExpansion}\\
&  =\sum_{uv\in E}\left(  \mathbf{a}_{u}-\mathbf{a}_{v}\right)  \left(
\mathbf{d}_{u,v}-\mathbf{d}_{v,u}\right)  ^{\dagger},
\end{align}
\end{subequations}
restricting in this case to graphs $G$ containing no isolated vertices.
We may then easily show that
\[
L(G)=\tfrac{1}{2}\bar{M}\bar{M}^{\dagger}.
\]
(The factor of $1/2$ may be seen to arise from doubling the edge-space by
introduction of arc vectors rather than edge vectors.)
We may thereby describe the Laplacian using incidence matrices, but without
reference to any particular orientation $\mathcal{F}$ of the edges.

Implicitly, the latter formulation of the Laplacian also describes a way in
which it may be formed as the partial trace of a rank-1 operator on
$\mathcal{H}_{V}\otimes\mathcal{H}_{E}$ which is determined by the graph $G$.
The rank-1 operator we may represent as an outer product $\psi_G\psi_G^{\dagger}$,
where $\psi_G\in\mathcal{H}_{V} \otimes\mathcal{H}_{E}$ is a vectorization of
the incidence matrix $\bar M$. Consider the vector
\begin{equation}
\label{eqn:incidenceVector}\psi_G=\tfrac{1}{\sqrt2}\sum_{uv\in
E}\left(  \mathbf{a}_{u}-\mathbf{a}_{v}\right)  \otimes\left(  \mathbf{d}%
_{u,v}-\mathbf{d}_{v,u}\right)  \in\mathcal{H}_{V}\otimes\mathcal{H}_{E};
\end{equation}
With the use of the partial trace operation, and letting $\mathbf{X}_{E(G)}$
be the characteristic function of $E(G)$, we can describe a precise
relationship between $\psi_G$ and the Laplacian:
\begin{equation}
\begin{aligned}[b]
	\text{tr}_{E}\left( \psi_G \psi_G^{\dagger}\right)
&=
	\sum_{u,v\in V} \Bigl( \boldsymbol{1}_{V}\otimes\mathbf{d}_{u,v}^{\dagger}\Bigr) \psi_G \psi_G^{\dagger}\Bigl( \boldsymbol{1}_{V}\otimes\mathbf{d}_{u,v}\Bigr) \label{eqn:nonNormPsiLaplacian}
\\&=
	\tfrac{1}{2}\!\sum_{u,v \in V} \mathbf{X}_{E(G)}(uv) \; (\mathbf{a}_u - \mathbf{a}_v) (\mathbf{a}_u - \mathbf{a}_v)^\dagger
\\&=
	\tfrac{1}{2}\!\sum_{u,v \in V} \mathbf{X}_{E(G)}(uv) \; (\mathbf{a}_u - \mathbf{a}_v) \,\mathbf{d}_{u,v}^\dagger \mathbf{d}_{u,v} \,(\mathbf{a}_u - \mathbf{a}_v)^\dagger	\;;
	\end{aligned}
\end{equation}
	Introducing separate terms involving $\mathbf{d}_{u,v}^\dagger \mathbf{d}_{u,v}$ and $\mathbf{d}_{v,u}^\dagger \mathbf{d}_{v,u}$ for each undirected edge $uv \in E(G)$, we may then obtain
\begin{equation}
	\begin{aligned}[b]
	\text{tr}_{E}\left( \psi_G \psi_G^{\dagger}\right)
	&=
	\tfrac{1}{2}\!\!\!\!\sum_{uv \in E(G)}\!\!\! \left[ (\mathbf{a}_u - \mathbf{a}_v) \,\mathbf{d}_{u,v}^\dagger \mathbf{d}_{u,v} \,(\mathbf{a}_u - \mathbf{a}_v)^\dagger \;+\; (\mathbf{a}_v - \mathbf{a}_u) \,\mathbf{d}_{v,u}^\dagger \mathbf{d}_{v,u} \,(\mathbf{a}_v - \mathbf{a}_u)^\dagger \right]
\\&=
	\tfrac{1}{2}\!\!\!\!\!\sum_{\;\;st,uv \in E(G)}\!\!\!\!\! \left[ (\mathbf{a}_s - \mathbf{a}_t) \otimes (\mathbf{d}_{s,t} - \mathbf{d}_{t,s})^\dagger\right] \left[ (\mathbf{d}_{u,v} - \mathbf{d}_{v,u}) \otimes (\mathbf{a}_u - \mathbf{a}_v)^\dagger\right]
\\[1ex]&=
	\tfrac{1}{2} \bar{M} \bar{M}^\dagger = L(G).
\end{aligned}\mspace{-20mu}
\end{equation}
Note that $\psi_G\in\mathcal{H}_{V} \otimes\mathcal{H}_{E}$, as it has been
defined above, may not be a unit vector; its normalization has been chosen
specifically so that $\| \psi_G \|_{2} = \sqrt{\text{tr}(L(G))} = \sqrt{2|E|}$, which will
differ from $1$. If we wish, we may renormalize it, and retain the relation
$L(G) = \text{tr}_{E}(\psi_G\psi_G^{\dagger})$ by dividing the Laplacian through by $2|E|$ to obtain an operator with unit trace.

Having obtained $L(G)$ as a partial trace of a rank-1 operator, we may ask the
following: what is the result of taking the other partial trace? This is (a
normalized version of) what is known in the literature as the \emph{edge
Laplacian},
\begin{equation}
L_{E}(G) = \text{tr}_{\text{V}} \left( \psi_G\psi_G^{\dagger}\right)  = \tfrac
{1}{2} \bar M^{\dagger}\bar M ,
\end{equation}
where the scalar of proportionality is determined by the normalization of
$\psi_G$; this can be shown by a similar development as in
Eqn.~\eqref{eqn:nonNormPsiLaplacian}. While less studied than the Laplacian, a
recent application of the edge Laplacian is in dynamic systems and the edge
agreement problem (see \cite{me}). The edge-Laplacian has the same positive
eigenvalues as $L(G)$, but as it (usually) acts on a much larger configuration
space (\emph{i.e.}\ when $G$ has more edges than vertices), it will have a
larger kernel, whose dimension is the size of a cycle-basis for $G$.
The above discussion can be summarized as follows:

\begin{proposition}
\label{prop:LaplacianMarginals}
Let $\psi_G$ be an incidence vector of a graph $G=\left(  V,E\right)  $, as
defined in Eqn.~\eqref{eqn:incidenceVector}. Then
\begin{equation}
L\left(  G\right)  =\text{\emph{tr}}_{E}\left(  \psi_G\psi_G^{\dagger}\right)
\qquad\text{and}\qquad L_{E}\left(  G\right)  =\emph{tr}_{V}\left(  \psi_G
\psi_G^{\dagger}\right)  .
\end{equation}

\end{proposition}

We may interpret the vector $\psi_G$ (or the renormalized version of this vector) as a quantum state vector on two systems of finite dimension, one of dimension at least $|V|$ and one of dimension at least $|E|$, supporting Hilbert spaces which subsume $\mathcal{H}_V$ and $\mathcal{H}_E$ respectively.
Each of these systems may themselves be composed of multiple subsystems, for instance spin-1/2 particles (\emph{i.e.}\ qubits), whose standard basis states are used to represent the indices $v \in V$ and $u,v \in V \times V$ for the basis vectors $\mathbf{a}_v$ and $\mathbf{d}_{u,v}$ respectively.
In the standard terminology of quantum information theory, $\psi_G$ is said to be a \emph{purification} of $L\left(  G\right)  $
and $L_{E}\left(  G\right)  $. These matrices are the \emph{reduced density
matrices} with respect to $\mathcal{H}_{E}$ and $\mathcal{H}_{V}$ --- albeit with the caveat that, as they are usually defined, $L(G)$ and $L_E(G)$ may have trace different from $1$.

On the basis of this observation we may apply the machinery of quantum theory.
The normalized Laplacian $\rho_G = \frac{1}{2|E|} L(G)$ may be then interpreted as a mixture of pure states (\emph{i.e.}\ a convex combination of rank-1 operators) can be given in terms of populations and coherences, by considering how $\psi_G$ may be interpreted as a linear combination of orthogonal vectors.
In the following, we will write $|v\rangle := \mathbf a_v \in \mathcal H_V$ for standard basis states in the vertex space, and $|u,v\rangle := \mathbf d_{u,v} \in \mathcal H_E$ for standard basis states in the edge space.
We may note that the expression for $\psi_G$ in Eq.~\eqref{eqn:incidenceVector} represents a linear combination of states of the form
\begin{equation}
		\psi_{uv}	\;=\;	\tfrac{1}{2} \Bigl[ |u\rangle - |v\rangle \Bigr] \otimes \Bigl[ |u,v\rangle - |v,u\rangle \Bigr]
\end{equation}
over all edges $uv \in E(G)$; by the orthogonality of $\beta_{uv} := |u,v\rangle - |v,u\rangle$ for distinct vertex-pairs $(u,v) \in V \times V$, the density operator $\rho_G = \tfrac{1}{2|E|} L(G)$ is a uniformly random mixture of operators
\begin{equation}
		\rho_G
	=
		\tfrac{1}{E(G)}\!\!\!
		\sum_{uv \in E(G)}\!\!\!
		\alpha_{uv} \alpha_{uv}^\dagger,
	\qquad
	\text{where $\alpha_{uv} = \tfrac{1}{\sqrt 2}|u\rangle - \tfrac{1}{\sqrt 2}|v\rangle$}.
\end{equation}
This suggests an interpretation of $\rho$ as a uniformly random mixture of pure states in the vertex-space, where each state in the mixture corresponds to a single edge of the graph.
The edge-vectors $\alpha_{uv} \in \mathcal H_V$ are not orthogonal vectors to one another when the edges are co-incident, and in that case would not be perfectly distinguishable from one another as quantum states.

For example, let us consider the graph $G=\{\{1,2,3\},\{\{1,2\},\{1,3\}\}\}$.
We denote the standard basis vectors of $\mathcal H_V$ by $|1\rangle$, $|2\rangle$, and $|3\rangle$ corresponding to the vertex labels.
To each edge $uv \in E$ we associate a unit vector
\begin{equation}
	|\{u,v\}\rangle \propto \alpha_{u}^{\{u,v\}}|u\rangle+\alpha_{v}^{\{u,v\}}|v\rangle \in \mathcal H_V	:
\end{equation}
the complex argument of the scalar of proportionality does not matter.
In our example,
\begin{equation}
|\{1,2\}\rangle=\alpha_{1}^{\{1,2\}}|1\rangle+\alpha_{2}^{\{1,2\}}|2\rangle
\qquad
\text{and}
\qquad
|\{1,3\}\rangle=\alpha_{1}^{\{1,3\}}|1\rangle+\alpha_{3}%
^{\{1,3\}}|3\rangle.
\end{equation}
By definition,%
\begin{equation}
|\alpha_{1}^{\{1,2\}}|^{2}+|\alpha_{2}^{\{1,2\}}|^{2}=|\alpha_{1}%
^{\{1,3\}}|^{2}+|\alpha_{2}^{\{1,3\}}|^{2}=1.
\end{equation}
A general state $\rho$ of the system expressing a statistical mixture of $|\{1,2\}\rangle$ and $|\{1,3\}\rangle$ is described by an operator
\begin{equation}
\rho = \omega_{\{1,2\}} |\{1,2\}\rangle\langle\{1,2\}|+\omega_{\{1,3\}} |\{1,3\}\rangle\langle\{1,3\}|,
\end{equation}
where $\omega_{\{1,2\}},\omega_{\{1,3\}} \ge 0$ and $\omega_{\{1,2\}} + \omega_{\{1,3\}} = 1$.
The operator $\rho$ can the be written as%
\begin{equation}
\rho=\left(
\begin{array}
[c]{ccc}%
\rho_{1,1} & \rho_{1,2} & \rho_{1,3}\\
\rho_{2,1} & \rho_{2,2} & \rho_{2,3}\\
\rho_{3,1} & \rho_{3,2} & \rho_{3,3}%
\end{array}
\right)  ,
\end{equation}
with%
\begin{align*}
	\rho_{1,1}
&=
	\omega_{\{1,2\}} \left|\alpha_{1}^{\{1,2\}}\right|\big.^{\!2}+\omega_{\{1,3\}}\left|\alpha_{1}^{\{1,3\}}\right|\big.^{\!2},
\\[1ex]
\rho_{1,2} &=\rho_{2,1}=\omega_{\{1,2\}}\alpha_{1}^{\{1,2\}}\overline
{\alpha}_{2}^{\{1,2\}},\\[1ex]
\rho_{1,3}  &  =\rho_{3,1}=\omega_{\{1,3\}}\alpha_{1}^{\{1,3\}}\overline
{\alpha}_{2}^{\{1,3\}},\\[1ex]
\rho_{2,2}  &  =\omega_{\{1,2\}}\left|\alpha_{2}^{\{1,2\}}\right|\big.^{\!2},\\[1ex]
\rho_{2,3}  &  =\rho_{3,2}=0,\\[1ex]
\rho_{3,3}  &  =\omega_{\{1,3\}}\left|\alpha_{2}^{\{1,3\}}\right|\big.^{\!2}.
\end{align*}
One conventionally interprets such an operator statistically with respect to a projective measurement process, where for some orthonormal basis $\{ \mathbf v_j : j = 1, \ldots, n \}$ the probability of outcome $j$ is $\mathbf v_j^\dagger \rho \mathbf v_j$, representing a realization of $\mathbf v_j$ as the state of the system.
For instance, setting $\mathbf v_j = |j\rangle$ represents a measurement of $\rho$ in the standard basis, and represents the population of the system which is in the state $|j\rangle$ (representing the vertex $j$ in this case) for a system initialized to the state $\rho$.
(Similarly, when measuring the state with a projective measurement with respect to the standard basis, $|\alpha_{1}^{\{1,2\}}|^{2}$ represents the probability of observing the state $|1\rangle$ in the ray $|\{1,2\}\rangle$ --- represented by the operator $|\{1,2\}\rangle\langle\{1,2\}|$ which is involved in the ensemble $\rho$.)

Therefore, $\rho_{i,i}$, with $i=1,2,3$, is the probability of the state
$|i\rangle$ in $\rho$. In other terms, if the same measurement is carried out
$N$ times (under the same initial conditions), $N\rho_{i,i}$ systems will be
observed in the state $|i\rangle$. (For this reason $\rho_{i,i}$ is sometimes said
to be the \emph{population} of $|i\rangle$.) Operationally, each $\rho_{i,i}$ is the
probability of getting the vertex $i$ when \textquotedblleft observing the
graph\textquotedblright, where the graph is itself represented by the state $\rho$.
(It must be remarked that the observation is performed with the respect to the standard basis; projective measurement involving other bases shall give superpositions of
vertices.) The cross terms of $\rho$ indicates the subsistence of a certain
amount of coherence in the system. In fact, $\rho_{i,j}$, with $i\neq j$,
expresses the coherence effects between the states $|i\rangle$ and
$|j\rangle$ arising from the presence of $\alpha_{uv}$ in the statistical mixture.

As we note above, when the mixture is equally weighted, and the states $|\{u,v\}\rangle$ are taken to be the vectors $\alpha_{uv}$, \emph{i.e.}\ uniform linear combinations up to a sign, we then obtain $\rho = \rho_G := \tfrac{1}{2|E|} L(G)$, the normalized Laplacian. Let
\begin{gather*}
\omega_{\{1,2\}}   =\omega_{\{1,2\}}=\tfrac{1}{2},\\[1ex]
\alpha_{1}^{\{1,2\}}   =\tfrac{1}{\sqrt{2}}\quad\text{ and }\quad\alpha_{2}%
^{\{1,2\}}=-\tfrac{1}{\sqrt{2}},\\[1ex]
\alpha_{1}^{\{1,3\}}   =\tfrac{1}{\sqrt{2}}\quad\text{ and }\quad\alpha_{2}%
^{\{1,3\}}=-\tfrac{1}{\sqrt{2}}.
\end{gather*}
Then,%
\begin{align*}
\rho &=
\tfrac{1}{2}|\{1,2\}\rangle\langle\{1,2\}|+\tfrac{1}{2}|\{1,3\}\rangle
\langle\{1,3\}|
=
\left(
\begin{array}
[c]{rrr}%
\frac{1}{2} & -\frac{1}{4} & -\frac{1}{4}\\[0.5ex]
-\frac{1}{4} & \frac{1}{4} & 0\\[0.5ex]
-\frac{1}{4} & 0 & \frac{1}{4}%
\end{array}
\right)
\\
&  =\tfrac{1}{4}\left[ \left(
\begin{array}
[c]{ccc}%
2 & 0 & 0\\
0 & 1 & 0\\
0 & 0 & 1
\end{array}
\right)  -\left(
\begin{array}
[c]{ccc}%
0 & 1 & 1\\
1 & 0 & 1\\
1 & 1 & 0
\end{array}
\right) \right]  =\frac{1}{2\left\vert E(G)\right\vert }\Bigl(
D(G)-A(G)\Bigr).
\end{align*}
Note that we may consider choosing $|\{u,v\}\rangle$ to be uniform superpositions with the same sign,
\begin{equation}
	|\{u,v\} \rangle = \varsigma_{uv}	:= \tfrac{1}{\sqrt 2}\Bigl(|u\rangle + |v\rangle\Bigr) ;
\end{equation}
like $\alpha_{uv}$, a standard basis measurement upon the state $\varsigma_{uv} \in \mathcal H_V$ would yield $u$ and $v$ with equal probability, and all other vertex labels with probability $0$.
The ensemble which arises from a uniform mixture of these unsigned edge-states is then
\begin{equation}
		\varrho_G
	=
		\tfrac{1}{E(G)}\!\!\!
		\sum_{uv \in E(G)}\!\!\!
		\varsigma_{uv} \varsigma_{uv}^\dagger
	=
		\tfrac{1}{2E(G)}\,
		\mathrm{tr}_E\!\left( \phi_G \phi_G^\dagger\right),
\end{equation}
where
\begin{equation}
  \phi_G	\;= \!\!	\sum_{uv \in E(G)} \!\!\varsigma_{uv} \otimes \Bigl( |u,v\rangle + |v,u\rangle \Bigr) \in \mathcal H_V \otimes \mathcal H_E	\;.
\end{equation}
By a similar analysis as that which demonstrates Proposition~\ref{prop:LaplacianMarginals}, we may then show that $\varrho_G$ is the normalized version of the \emph{signless Laplacian}, $L^{+}(G)=D(G)+A(G)$.
The operators $\rho_G$ and $\varrho_G$ are therefore density matrices in the cases where the vertex-states have equal weighting in the pure states $|\{u,v\}\rangle$, and where each of these edge-states have equal weighting in the mixture over edge-states.
Among other reasons, $\rho_G$ is often preferred to $\varrho_G$ because $\rho_G$ has an all-ones eigenvector corresponding to the zero eigenvalue.

\bigskip

By interpreting $\psi_G$ as a pure state, we may apply the ideas of quantum
information in the graph-theoretic framework. We consider below some
directions. The \emph{Schmidt rank} of a vector $\mathbf{v} \in\mathcal{H}%
_{A}\otimes\mathcal{H}_{B}$ where $\mathcal{H}_{A}\cong\mathbb{C}^{A}$ and
$\mathcal{H}_{B}\cong\mathbb{C}^{B}$ ($|A|=n$, $|B|=m$ and $m\geq n$) is
defined as the minimum number of coefficients $\alpha_{i} > 0$ such that
\begin{equation}
\mathbf{v}=\sum_{i=1}^{m}\alpha_{i}(\mathbf{e}_{i}\otimes\mathbf{f}_{i}),
\end{equation}
where $\{\mathbf{e}_{i}:i=1,...,m\}$ and $\{\mathbf{f}_{i}:i=1,...,m\}$ are
some pair of orthonormal bases. As $n \le m$, it follows that the Schmidt rank
is no larger than $n$. The scalars $\alpha_{i}$ are referred to as
\emph{Schmidt coefficients}. It is simple to show that the Schmidt rank of
$\mathbf{v}$, denoted by $\text{rank}_{S}(\psi_G)$, is equal to the
rank of its partial traces; this follows from a direct relationship between
$\mathbf{v}$ and a transformation $V : \mathcal{H}_{B} \to\mathcal{H}_{A}$
defined through its singular value decomposition,
\begin{equation}
V = \sum_{i=1}^{m}\alpha_{i} \; \mathbf{e}_{i} \mathbf{f}_{i}^{\dagger}.
\end{equation}
It is well-known (a consequence of the matrix-tree theorem) that the rank of
the Laplacian of a graph $G$ on $n$ vertices is equal to $n-w(G)$, where
$w(G)$ is the number of connected components of $G$. Directly from the definitions:

\begin{proposition}
Let $\psi_G$ be an incidence vector of a graph $G$ on $n$ vertices. Then, the
Schmidt rank of $\psi_G$ is
\begin{equation}
\textup{rank}_{S}(\psi_G) = \textup{ rank}(L(G))=n-w(G).
\end{equation}
\end{proposition}

We say that two vectors $\mathbf{v}, \mathbf{w }\in\mathcal{H}_{A}
\otimes\mathcal{H}_{B}$ are \emph{locally unitarily equivalent} (or
LU-equivalent) if there exist unitary operators $U: \mathcal{H}_{A}
\to\mathcal{H}_{A}$ and $V: \mathcal{H}_{B} \to\mathcal{H}_{B}$ such that
$\mathbf{w }= (U \otimes V) \mathbf{v}$. By considering the Schmidt
decompositions of two such vectors, it is clear that $\mathbf{v}$ and
$\mathbf{w}$ are LU-equivalent if and only if they have the same Schmidt
coefficients.
Let us denote by Sp$(G)$ the \emph{spectrum} of a graph $G$, which we define
as the ordered sequence of eigenvalues of the Laplacian of $G$.

\begin{proposition}
Let $\psi_G$ and $\psi_H$ be incidence vectors of two graphs $G$ and $H$,
respectively, having the same number of vertices and edges. Then $\psi_G$ and
$\psi_H$ are LU-equivalent if and only if $\textup{Sp}(G)=\textup{Sp}(H)$.
\end{proposition}

\begin{proof}
If $\psi_G$ and $\psi_H$ are LU-equivalent, it follows that there is a unitary
$U : \mathcal{H}_{V} \to\mathcal{H}_{V}$ for which
\begin{equation*}
	L(H) =
	\text{tr}_{E}(\psi_H\psi_H^{\dagger}) = U\,\text{tr}_{E}(\psi_G\psi_G^{\dagger
	})\,U^{\dagger}= U L(G) U^{\dagger}
\end{equation*}
so that $\textup{Sp}(G)=\textup{Sp}(H)$.  Conversely, if
$\textup{Sp}(G)=\textup{Sp}(H)$, we have $\text{tr}_{E}(\psi_H\psi_H
^{\dagger}) = L(H) = U L(G) U^{\dagger}= U\, \text{tr}_{E}(\psi_G\psi_G^{\dagger})
\, U^{\dagger}$ for some unitary $U$.  By considering the Schmidt
decompositions of $\psi_G$ and $\psi_H$, it follows that there exists a unitary
$V: \mathcal{H}_{E} \to\mathcal{H}_{E}$ such that $\psi_H= (U \otimes V)
\psi_G$.  Thus $\psi_G$ and $\psi_H$ are LU-equivalent.  \hfill
\end{proof}

\bigskip

This means that graphs which are Laplacian cospectral correspond to local unitarily equivalent incidence vectors.

Recall that the edge-states $\alpha_{uv}, \alpha_{vw}$ are not perfectly distinguishable from one another through any projective measurement, for any pair of edges $uv, vw \in E$ which coincide, as these vectors are not orthogonal.
In particular, for measurement in the standard basis, there is a probability of $\tfrac{1}{2}$ that any such edge-state will give rise to the common vertex $w$, which is perfectly ambiguous when attempting to distinguish $uv$ from $vw$.
Thus, despite being a uniformly random mixture of the edge-states, the imperfect distinguishability of the edge-states implies that the \emph{von Neumann entropy} of $\rho_G$,
\begin{equation}
  S(\rho_G)	\;=\;	- \mathrm{tr}\Bigl(\rho_G \ln(\rho_G)\Bigr)	\;=\;	-\!\!\!\!\sum_{\lambda \in \mathrm{Spec}(\rho_G)} \!\!\!\!\lambda \ln(\lambda),
\end{equation}
indicates something of the structure of the graph with respect to coincidence of edges.
In particular, as the Laplacian $L(G)$ and the edge-Laplacian $L_E(G)$ have the same spectrum of non-zero eigenvalues, we have
\begin{equation}
	S_{V}(\psi_G) := -\text{tr}(L(G)\ln L(G))=-\text{tr}(L_{E}(G)\ln L_{E}(G)).
\end{equation}
The quantity $S_{V}(\psi_G)$ has been recently studied in several contexts (see, \emph{e.g.}, \cite{han}
for an application in pattern recognition and \cite{ro} for an application in
loop quantum gravity). Another term for $S_A(\mathbf{\psi)}$, for arbitrary pure states $\psi \in \mathcal H_A \otimes \mathcal H_B$, is \emph{entropy of entanglement},
because $S_A(\psi)$ quantifies the amount of entanglement between subsystems with Hilbert space $\mathcal{H}_{A}$ and $\mathcal{H}_{V}$.
Entropy of entanglement is indeed \emph{the} asymptotic entanglement measure for bipartite pure states.
Hence the next fact, giving an interpretation to $S_{V}(\psi_G)$ (or, equivalently, $S(G)$):

\begin{proposition}
The von Neumann entropy of $\rho_G = \tfrac{1}{2|E|} L(G)$ is the amount of
entanglement between the subsystems of $\psi_G$ corresponding to vertices (with Hilbert space $\mathcal{H}_{V}$) and edges (with Hilbert space $\mathcal{H}_{E}$), respectively.
\end{proposition}

Two graphs $G$ and $H$ are \emph{isomorphic} if there is a bijection
$f:V(G)\longrightarrow V(H)$ such that $\{i,j\}\in E(G)$ if and only if
$\{f(i),f(j)\}\in E(H)$. A \emph{permutation matrix} is a matrix with entries
in the set $\{0,1\}$ and a unique $1$ entry in each row and column. Then, two
graphs $G$ and $H$ are isomorphic if and only if there is a permutation matrix
$P$ such that
\begin{equation}
P \,\rho_G  \,P^{T}= \rho_H  .
\end{equation}
In this case, the matrices $\rho_G$ and $\rho_H$ are said to be \emph{permutation congruent}. Two permutation congruent
matrices have the same eigenvalues and so two isomorphic graphs share the same
spectrum, $\textrm{Sp}(G) = \textrm{Sp}(H)$. It is well known that the converse is does not hold,
\emph{i.e.}\ there are many non-isomorphic graphs which share the same spectrum. Such
pairs of graphs are called \emph{cospectral}. 

\begin{proposition}
Given two graphs with the same number of vertices. Then there exist graphs $G$
and $H$ with $S(G)=S(H)$ but \emph{Sp}$(G)\neq$ \emph{Sp}$(H)$.
\end{proposition}

\bigskip
We call two graphs \emph{coentropic} if their Laplacians have the same von Neumann entropy. It is clear that the number of equivalence classes of coentropic graphs with $n$ vertices and $m$ edges is a lower bound to the number of (pure) bipartite entanglement classes with subsystems of corresponding dimension.
\bigskip

\begin{proof}
It is clear that if Sp$(G)=$ Sp$(H)$ then $S(G)=S(H)$, since $S(G)$ and $S(H)$ are determined by the spectra. However, the following two graphs have spectra $[0,3,3,3,3,6,8,8]/34$ and $[0,2,2,4,6,6,6,8]/34$ and equal von
Neumann entropy $S(G)=S(H)=\ln(34)-[18\ln(3)+54\ln(2)]/34$: \\
\begin{minipage}[b]{0.48\linewidth}%
\centering
\includegraphics[width=0.6\linewidth]{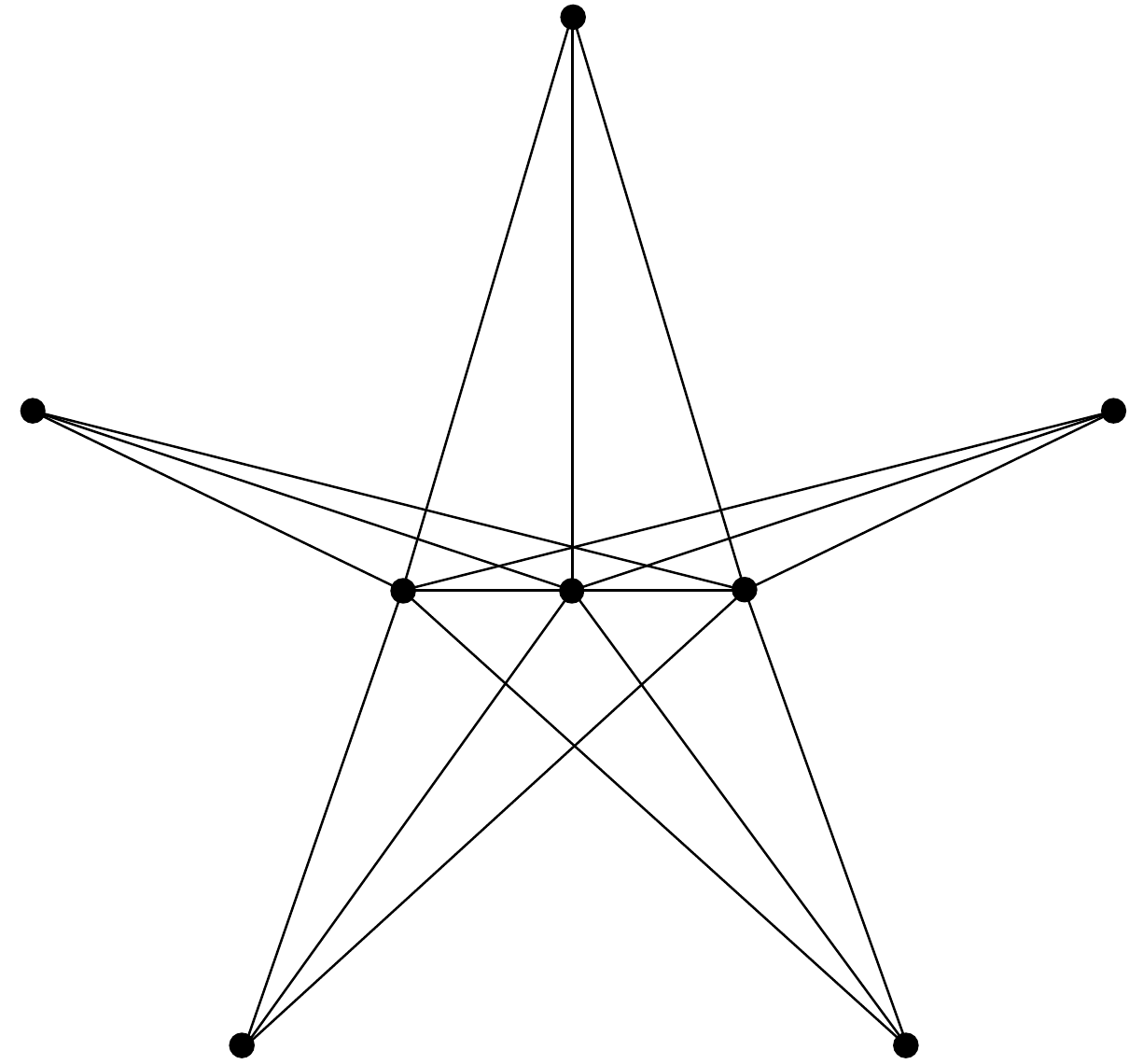}
\end{minipage}
\begin{minipage}[b]{0.48\linewidth}
\centering
\includegraphics[width=0.6\linewidth]{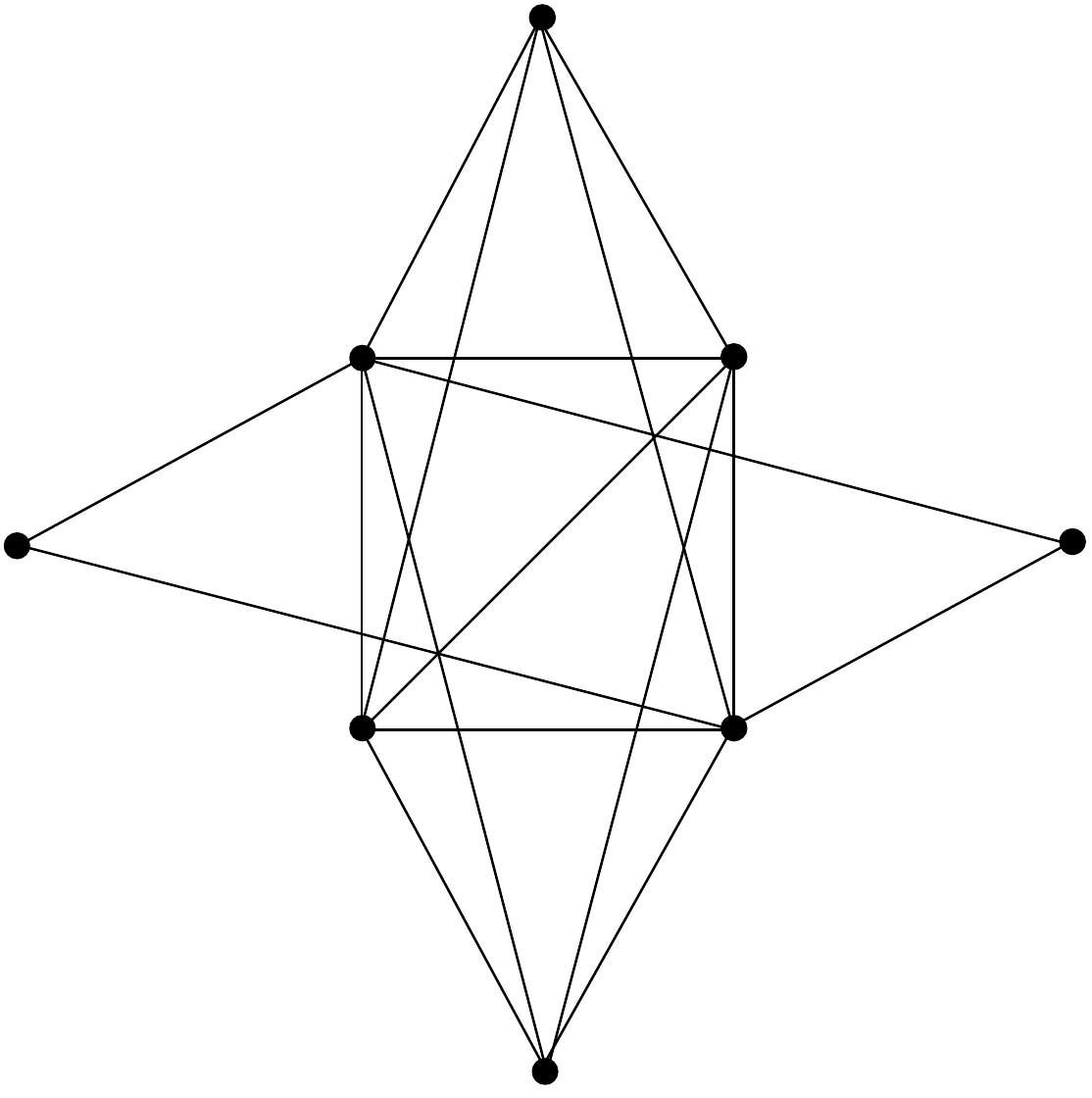}
\end{minipage}
\end{proof}

\bigskip

There are two pairs of non-isomorphic and non-cospectral graphs of size $8$ with the same
entropy. There are 8 such pairs of size 9 (enumerated in the table below) and 76 pairs of size 10.
In all these cases, the pairs share the same number of edges. In this case, it sufficient that the entropy
of the un-normalized Laplacian coincides:
\[
\hat{S}(G) \;=\; -\!\!\!\!\sum_{\lambda\in\text{Sp}(G)}\!\!\!\!\lambda\ln\lambda
\]
where Sp$(G)$ is here the spectrum of the un-normalized Laplacian. No examples are
known for pairs with different numbers of edges.

\begin{table}
\begin{tabular}{l|l}
Graph & Entropy \\
\hline
\hline
\footnotesize \{\{1, 8\}, \{1, 9\}, \{2, 8\}, \{2, 9\}, \{3, 8\}, \{3, 9\}, \{4, 8\}, \{4, 9\},    &  $-\frac{1}{5}\ln(3)$\\
\footnotesize \hspace{10pt} \{5, 8\}, \{5, 9\}, \{6, 8\}, \{6, 9\}, \{7, 8\}, \{7, 9\}, \{8, 9\}\} & \hspace{10pt}$+\frac{3}{5}\ln(2)$\\
\footnotesize \{\{1, 7\}, \{1, 8\}, \{1, 9\}, \{2, 7\}, \{2, 8\}, \{2, 9\}, \{3, 7\}, \{3, 8\},    & \hspace{10pt}$+\ln(5)$\\
\footnotesize \hspace{10pt} \{3, 9\}, \{4, 9\}, \{5, 9\}, \{6, 9\}, \{7, 8\}, \{7, 9\}, \{8, 9\}\} & \\
\hline
\footnotesize \{\{1, 7\}, \{1, 8\}, \{1, 9\}, \{2, 7\}, \{2, 8\}, \{2, 9\}, \{3, 7\}, \{3, 8\}, \{3, 9\}, \{4, 7\}, & $-\frac{15}{19}\ln(3)$ \\
\footnotesize \hspace{10pt}  \{4, 8\}, \{4, 9\}, \{5, 7\}, \{5, 8\}, \{5, 9\}, \{6, 9\}, \{7, 8\}, \{7, 9\}, \{8, 9\}\} &\hspace{10pt}$-\frac{5}{19}\ln(2)$ \\
\footnotesize \{\{1, 6\}, \{1, 7\}, \{1, 8\}, \{1, 9\}, \{2, 6\}, \{2, 7\}, \{2, 8\}, \{2, 9\}, \{3, 8\}, \{3, 9\}, & \hspace{10pt}$+\ln(19)$\\
\footnotesize \hspace{10pt}  \{4, 8\}, \{4, 9\}, \{5, 9\}, \{6, 7\}, \{6, 8\}, \{6, 9\}, \{7, 8\}, \{7, 9\}, \{8, 9\}\} & \\
\hline

\footnotesize \{\{1, 5\}, \{1, 8\}, \{1, 9\}, \{2, 6\}, \{2, 8\}, \{2, 9\}, \{3, 7\}, \{3, 8\}, \{3, 9\}, \{4, 8\}, &\\
\footnotesize \hspace{10pt}  \{4, 9\}, \{5, 8\}, \{5, 9\}, \{6, 8\}, \{6, 9\}, \{7, 8\}, \{7, 9\}, \{8, 9\}\} &   $\ln(3)$\\
\footnotesize \{\{1, 7\}, \{1, 8\}, \{1, 9\}, \{2, 7\}, \{2, 8\}, \{2, 9\}, \{3, 7\}, \{3, 8\}, \{3, 9\}, \{4, 7\}, & \hspace{10pt}$+\frac{7}{6}\ln(2)$\\
\footnotesize \hspace{10pt}  \{4, 8\}, \{4, 9\}, \{5, 7\}, \{5, 8\}, \{5, 9\}, \{6, 9\}, \{7, 9\}, \{8, 9\}\} & \\
\hline

\footnotesize \{\{1, 6\}, \{1, 7\}, \{1, 8\}, \{1, 9\}, \{2, 6\}, \{2, 7\}, \{2, 8\}, \{2, 9\}, \{3, 8\}, \{3, 9\}, &  \\
\footnotesize \hspace{10pt}  \{4, 8\}, \{4, 9\}, \{5, 9\}, \{6, 7\}, \{6, 8\}, \{6, 9\}, \{7, 8\}, \{7, 9\}\} &1.91025843 \\
\footnotesize \{\{1, 7\}, \{1, 8\}, \{1, 9\}, \{2, 7\}, \{2, 8\}, \{2, 9\}, \{3, 7\}, \{3, 8\}, \{3, 9\}, \{4, 7\}, & \\
\footnotesize \hspace{10pt}  \{4, 8\}, \{4, 9\}, \{5, 7\}, \{5, 8\}, \{5, 9\}, \{6, 9\}, \{7, 8\}, \{7, 9\}\} & \\
\hline

\footnotesize \{\{1, 6\}, \{1, 7\}, \{1, 8\}, \{1, 9\}, \{2, 6\}, \{2, 7\}, \{2, 8\}, \{2, 9\}, \{3, 8\}, \{3, 9\}, & $\frac{47}{20}\ln(2)$ \\
\footnotesize \hspace{10pt}  \{4, 8\}, \{4, 9\}, \{5, 8\}, \{5, 9\}, \{6, 7\}, \{6, 8\}, \{6, 9\}, \{7, 8\}, \{7, 9\}, \{8, 9\}\} & \hspace{10pt}$-\frac{6}{5}\ln(3)$\\
\footnotesize \{\{1, 7\}, \{1, 8\}, \{1, 9\}, \{2, 7\}, \{2, 8\}, \{2, 9\}, \{3, 7\}, \{3, 8\}, \{3, 9\}, \{4, 7\}, & \hspace{10pt}$+\ln(5)$\\
\footnotesize \hspace{10pt}  \{4, 8\}, \{4, 9\}, \{5, 7\}, \{5, 8\}, \{5, 9\}, \{6, 8\}, \{6, 9\}, \{7, 8\}, \{7, 9\}, \{8, 9\}\} & \\
\hline

\footnotesize \{\{1, 6\}, \{1, 7\}, \{1, 8\}, \{1, 9\}, \{2, 6\}, \{2, 7\}, \{2, 8\}, \{2, 9\}, \{3, 8\}, \{3, 9\}, & $\frac{6}{19}\ln(2)$ \\
\footnotesize \hspace{10pt}  \{4, 8\}, \{4, 9\}, \{5, 8\}, \{5, 9\}, \{6, 7\}, \{6, 8\}, \{6, 9\}, \{7, 8\}, \{7, 9\}\} & \hspace{10pt}$-\frac{7}{38}\ln(7)$\\
\footnotesize \{\{1, 7\}, \{1, 8\}, \{1, 9\}, \{2, 7\}, \{2, 8\}, \{2, 9\}, \{3, 7\}, \{3, 8\}, \{3, 9\}, \{4, 7\}, & \hspace{10pt}$-\frac{15}{19}\ln(3)$\\
\footnotesize \hspace{10pt}  \{4, 8\}, \{4, 9\}, \{5, 7\}, \{5, 8\}, \{5, 9\}, \{6, 7\}, \{6, 8\}, \{7, 9\}, \{8, 9\}\} & \hspace{10pt}$+\ln(19)$\\
\hline

\footnotesize \{\{1, 4\}, \{1, 5\}, \{1, 7\}, \{1, 8\}, \{1, 9\}, \{2, 6\}, \{2, 9\}, \{3, 6\}, \{3, 9\}, \{4, 5\}, & $\frac{43}{20}\ln(2)$ \\
\footnotesize \hspace{10pt}  \{4, 7\}, \{4, 8\}, \{4, 9\}, \{5, 7\}, \{5, 8\}, \{5, 9\}, \{6, 9\}, \{7, 8\}, \{7, 9\}, \{8, 9\}\} &\hspace{10pt}$-\frac{21}{20}\ln(3)$ \\
\footnotesize \{\{1, 5\}, \{1, 8\}, \{1, 9\}, \{2, 6\}, \{2, 7\}, \{2, 8\}, \{2, 9\}, \{3, 6\}, \{3, 7\}, \{3, 8\}, & \hspace{10pt}$+\ln(5)$\\
\footnotesize \hspace{10pt}  \{3, 9\}, \{4, 8\}, \{4, 9\}, \{5, 8\}, \{5, 9\}, \{6, 8\}, \{6, 9\}, \{7, 8\}, \{7, 9\}, \{8, 9\}\} & \\
\hline

\footnotesize \{\{1, 4\}, \{1, 6\}, \{1, 7\}, \{1, 8\}, \{1, 9\}, \{2, 5\}, \{2, 6\}, \{2, 7\}, \{2, 8\}, \{2, 9\}, &  \\
\footnotesize \hspace{10pt}  \{3, 9\}, \{4, 6\}, \{4, 7\}, \{4, 8\}, \{4, 9\}, \{5, 6\}, \{5, 7\}, \{5, 8\}, \{5, 9\}, \{6, 8\},& \\
\footnotesize \hspace{10pt}  \{6, 9\}, \{7, 8\}, \{7, 9\}, \{8, 9\}\} & $\frac{59}{24}\ln(2)$ \\
\footnotesize \{\{1, 6\}, \{1, 7\}, \{1, 8\}, \{1, 9\}, \{2, 6\}, \{2, 7\}, \{2, 8\}, \{2, 9\}, \{3, 6\}, \{3, 7\}, & \hspace{10pt}$+\frac{1}{4}\ln(3)$\\
\footnotesize \hspace{10pt} \{3, 8\}, \{3, 9\}, \{4, 6\}, \{4, 7\}, \{4, 8\}, \{4, 9\}, \{5, 8\}, \{5, 9\}, \{6, 7\}, \{6, 8\},& \\
\footnotesize \hspace{10pt}  \{6, 9\}, \{7, 8\}, \{7, 9\}, \{8, 9\}\} & \\
\hline

\end{tabular}
\end{table}

\end{document}